\documentclass{article}
\usepackage[utf8]{inputenc}
\usepackage{amsmath, amssymb, amsthm}
\usepackage{graphicx}
\usepackage{geometry}
\usepackage{xcolor}
\usepackage{bbm}

\newtheorem{lemma}{Lemma}
\newtheorem{proposition}[lemma]{Proposition}

\newtheorem{theorem}{Theorem}
\newtheorem{corollary}[lemma]{Corollary}
\newtheorem*{remark}{Remark}

\newcommand{\N}{\mathbb{N}}
\newcommand{\R}{\mathbb{R}}

\newcommand{\C}{\mathbb{C}}

\newcommand{\E}{\mathbb{E}}
\newcommand{\Var}{\mathbb{V}\mathrm{ar}}
\newcommand{\Cov}{\C\mathrm{ov}}
\newcommand{\ov}[1]{\overline{#1}}

\title{Central moments of the free energy of the O'Connell-Yor polymer}
\author{Christian Noack\thanks{Department of Mathematics, Cornell University. C.N. was supported by NSF RTG grant 1645643.} \and Philippe Sosoe\thanks{Department of Mathematics, Cornell University. P.S.'s research was partially supported by NSF grant DMS 1811093.} }
\date{\today}

\begin{document}
\maketitle

\begin{abstract}
    Sepp\"al\"ainen and Valk\'o showed in \cite{SV} that for a suitable choice of parameters, the variance growth of the free energy of the stationary O'Connell-Yor polymer is governed by the exponent $2/3$, characteristic of models in the KPZ universality class.
    
    We develop exact formulas based on Gaussian integration by parts to relate the cumulants of the free energy, $\log Z_{n,t}^\theta$,  to expectations of products of quenched cumulants of the time of the first jump from the boundary into the system, $s_0$. We then use these formulas to obtain estimates for the $k$-th central moment of $\log Z_{n,t}^\theta$ as well as the $k$-th annealed moment of $s_0$ for $k> 2$, with nearly optimal exponents $(1/3)k+\epsilon$ and $(2/3)k+\epsilon$, respectively.
\end{abstract}

\section{Introduction}
The semi-discrete polymer in a Brownian environment was introduced by O'Connell and Yor in \cite{OY}. It is one of only a few known examples of integrable polymer models.  To define it, let $n\ge 1$, $t>0$, and $B_n(t)$, $n=1,2,\dots$ be independent Brownian motions.   Introduce the energy
\begin{equation*}
\mathcal{E}_{n,t}(s_1,\ldots,s_{n-1})=\sum_{j=1}^n(B_j(s_{j})-B_j(s_{j-1})),
\end{equation*}
where we set $s_0:=0$ and $s_n:=t$. The semi-discrete (point-to-point) polymer partition function from $(0,0)$ to $(t,n)$ is given by
\begin{equation*}
Z_{n,t}=\int_{0<s_1<\cdots<s_{n-1}<t}e^{ \mathcal{E}_{n,t}(s_1,\ldots,s_{n-1})}\,\mathrm{d}s_1\cdots\mathrm{d}s_{n-1}.
\end{equation*}
The probabilistic interpretation of the right-hand side is as a Gibbs ensemble of up-right paths between $(0,0)$ and $(t,n)$. Each path consists of $n-1$ Poisson-distributed successive jumps at times $0<s_1<s_2,\ldots< s_{n-1}< t$ of height one between discrete levels $j=1,\ldots,n$. For each $j$, the path remains on level $j$ for time $s_j-s_{j-1}$. See \cite[Definition 1.1]{BCF} for a precise description of the path interpretation. The path interpretation justifies the name \emph{polymer}, and reveals $Z_{n,t}$ as the partition function of the Gibbs ensemble described above.

In this paper, we consider a family of stationary versions of the polymer partition function,  also studied in \cite{OY}. To define it, we introduce an extra two-sided Brownian motion $B_0(s)$, $s\in\mathbb{R}$, independent of $B_1,\ldots, B_n$ and also extend the Brownian motions $B_1,\ldots, B_n$ to two-sided Brownian motions. For $\theta>0$, define
\[\mathcal{E}^\theta_{n,t}(s_0,\cdots, s_{n-1}):=\theta s_0-B_0(s_0)+B_1(s_0)+\mathcal{E}_{n,t}(s_1,\cdots,s_{n-1}).\]
The stationary partition function is then
\begin{equation*}Z^\theta_{n,t}=\int_{-\infty<s_0<s_1<\cdots< s_{n-1}<t}e^{\mathcal{E}^\theta_{n,t}(s_0,\cdots,s_{n-1})}\,\mathrm{d}s_0\mathrm{d}s_1\cdots \mathrm{d}s_{n-1}.
\end{equation*}
For $n=0$, we let
\[Z_{0,t}^\theta=e^{-B_0(t)+\theta t}.\]
Note that now the $s_j$ can now range over the entire real line. Following Sepp\"al\"ainen and Valk\'o \cite{SV}, the Gibbs distribution of the initial jump $s_0$ plays a key role in the analysis in this paper, because it is a dual variable to the parameter $\theta>0$.

The main result in \cite{OY} implies that the free energy, $\log Z^\theta_{n,t}$, equals a combination of a sum of i.i.d.\ random variables and the Brownian motion $B_0(t)$:
\begin{proposition}[\cite{OY}]\label{prop: burke}
For each $n\ge 1$ and $t\geq0$, we have the identity
\begin{equation}\label{eqn: burke}
\log Z_{n,t}^\theta = \sum_{j=1}^n r_j^\theta(t)-B_0(t)+\theta t
\end{equation}
where
\[r_j^\theta(t):=\log Z_{j,t}^\theta-\log Z_{j-1,t}^\theta\]
are independent and identically distributed, with law equal to that of the random variable 
\[\log \frac{1}{X_\theta},\]
where $X_\theta$ is gamma-distributed with parameter $\theta$:
\[\mathbb{P}(X_\theta\in \mathrm{d}x)=\frac{1}{\Gamma(\theta)}x^{\theta-1}e^{-x}\mathrm{d}x,\]
where $\Gamma$ denotes the Gamma function, see \eqref{eqn: gamma-def}.
\end{proposition}

O'Connell and Moriarty \cite{MO} used the representation \eqref{eqn: burke} of Proposition \ref{prop: burke}, to compute the first order asymptotics of $\log Z_{n,t}$. Since its introduction in \cite{OY}, the semi-discrete polymer has been the subject of much investigation, revealing a rich algebraic structure far beyond the invariant measure statement contained in Proposition \ref{prop: burke}. See for example \cite{BC, BCF, CN, IS, J, JM, MSV, MO, OY, O, SV}.  Here we mention only a few of the many existing results about the semi-discrete polymer. In \cite{O}, O'Connell embedded the processes $\log Z_{j,t},\, j=1,\ldots,n$, $t>0$ in a triangular array of solutions to stochastic differential equations. He identified $\log Z_{n,t}$, as the first coordinate of an $n$-dimensional diffusion, the  $h$-transform of a Brownian motion by a certain Whittaker function. O'Connell used this connection to obtain an explicit formula for the Laplace transform of $\log Z_{n,t}$. Borodin, Corwin, and Ferrari \cite{BCF} used a modification of O'Connell's formula to show that the centered and rescaled free energy $\log Z_{n,t}$ converges in distribution to a Tracy-Widom GUE random variable.

Closer to the spirit of this paper, Sepp\"al\"ainen and Valk\'o adapted an argument from Sepp\"al\"ainen's work on the discrete log-gamma polymer  \cite{S} to obtain upper and lower bounds for the fluctuation exponents associated with the polymer. Predictions from physics \cite{KS} have led to the expectation that, for a broad family of 1+1-dimensional polymer models in random environments, there exist  exponents $\chi$, $\xi$ such that the variance of the free energy is of order $n^{2\chi}$, while the typical deviation of the polymer paths from a straight line is of order $n^\xi$. For the stationary semi-discrete polymer, the paper \cite{SV} contains a proof of the estimates
\begin{equation}\label{eqn: SV}
\begin{split}
    \Var(\log Z_{n,t}^\theta)&\asymp n^{2\chi},\\
    \mathbb{E}[E_{n,t}^\theta[|s_0|]]&\asymp n^{\xi},
\end{split}
\end{equation}
with $\xi=2\chi=\frac{2}{3}$, where $E_{n,t}^\theta[\cdot]$ denotes the expectation with respect to the (random) polymer measure (see Definition \ref{eqn: quenched-E}). See also Moreno-Flores, Sepp\"al\"ainen, and Valk\'o \cite{MSV} for a derivation of the fluctuation and wandering exponents in the so-called \emph{intermediate disorder regime} where the partition function $Z_{n,t}^\theta$ has an additional $n$-dependent temperature parameter. In Section \ref{sec: convexity}, we reprove the upper bounds of \eqref{eqn: SV} by an alternative argument using the convexity of the free energy, $\log Z_{n,t}^\theta$ in the parameter $\theta$.

Our main result complements the upper bounds in \eqref{eqn: SV} with nearly optimal (up to $n^\epsilon$) estimates for all central moments of $\log Z_{n,t}^\theta$ and all annealed moments of $s_0$, implying strong concentration on an almost optimal scale. As explained in Section \ref{sec: estimates}, the proof relies on inequalities that appear closely related to the predicted Kardar-Parisi-Zhang scaling relations \cite{C,KS}.

It may be possible to extend the argument to certain stationary integrable models such as the log-gamma polymer \cite{S}, the strict-weak polymer \cite {CSS2015,OO2015}, the beta polymer \cite{BC2016}, and the inverse-beta polymer \cite{TL2015}.  It may in fact be possible to extend the argument simultaneously to these four polymers using the Mellin-transform framework put forth in \cite{CN2018}.  We leave such potential extensions to later work.

\subsection{Main Results}
To state our results, we introduce some notation for expectations with respect to the Gibbs measure associated with $\log Z_{n,t}^\theta$. Let $\theta>0$, $n\ge 1$, $t>0$, and $f=f(s_0,\cdots,s_{n-1})$ be a real-valued function on $\mathbb{R}^n$ such that 
\begin{equation}\label{eqn: f-condition}
    |f(s_0,\cdots, s_{n-1})|\le e^{-\nu \min(s_0,0)}\quad \text{ for all } s_0\in \R
\end{equation}
with some $\nu<\theta$.

We define the \emph{quenched expectation} by 
\begin{equation}\label{eqn: quenched-E}
    E_{n,t}^\theta[f]:=\frac{1}{Z_{n,t}^\theta}\int_{-\infty<s_0<s_1<\cdots<s_{n-1}<t}e^{\mathcal{E}_{n,t}^\theta(s_0,\cdots,s_{n-1})}f(s_0,s_1\ldots,s_{n-1})\,\underline{\mathrm{d}s},
\end{equation}
where
\[\underline{\mathrm{d}s}=\mathrm{d}s_0\,\mathrm{d}s_1\cdots \mathrm{d}s_{n-1}.\]
The \emph{annealed expectation} is defined by
\begin{equation*}
    \mathbb{E}^\theta_{n,t}[f]:=\mathbb{E}[E_{n,t}^\theta[f]].
\end{equation*}
In many instances below, $n$ and $t$ are fixed throughout a section or computation, and we omit these variables from the notation: $\mathbb{E}^\theta[f]=\mathbb{E}^\theta_{n,t}[f]$. 

Let $\mathbbm{1}_A$ be the indicator of a set $A\subset \mathbb{R}^n$:
\begin{equation*}
    \mathbbm{1}_A(s_0,\ldots, s_{n-1})=\begin{cases}
        1& \quad \text{ if } (s_0,\ldots,s_{n-1})\in A,\\
        0&\quad \text{ otherwise.}
    \end{cases}
\end{equation*}
We use the suggestive notation
\begin{equation*}
    P^\theta_{n,t}(A):= E^\theta_{n,t}[\mathbbm{1}_A]\quad \text{ and } \quad 
    \mathbb{P}^\theta(A):=\mathbb{E}^\theta[\mathbbm{1}_A].
\end{equation*}
We refer to the first quantity as the quenched probability of the event $A$, and the second quantity as its annealed probability.

Our main result provides near-optimal estimates for \emph{any} moment of the centered free energy and \emph{any} annealed moment of the time of first jump:
\begin{theorem}\label{thm: main-moments} 
Let $\psi_1(\theta)=\frac{\mathrm{d}}{\mathrm{d}\theta}\Gamma'(\theta)/\Gamma(\theta)$ denote the \emph{trigamma function}, and suppose that
\begin{equation}\label{eqn: char-dir-1}
    |t-n\psi_1(\theta)|\le An^{2/3}.
\end{equation}
For every $\epsilon>0$, $\theta\in (0,\infty)$, and $p\in (0,\infty)$, there exists a constant $C=C(\epsilon,\theta,p)>0$ such that      for all  $n\in \N$,   
\begin{align}
 \mathbb{E}[|\overline{\log Z_{n,t}^\theta}|^p]&\le Cn^{(1/3)p+\epsilon} \quad \text{and}\label{eq: Thrm 1-1}\\ 
 \mathbb{E}^\theta_{n,t}[|s_0|^p]&\le C n^{(2/3)p+\epsilon}\label{eq: Thrm 1-2}
\end{align}
where $\ov{X}=X-\mathbb{E}[X]$ denotes the centered random variable. 
\end{theorem}
This result should be compared to that in \cite{SV}, where the following bounds were obtained for the corresponding moments
\begin{equation}\label{eqn: SV-moment}
\begin{split}
    \mathbb{E}[|\ov{\log Z_{n,t}^\theta}|^{p}]&\le C(\theta,p)n^{(1/3)p} \quad 0<p<3 \\
 \mathbb{E}_{n,t}^\theta [|s_0|^p]&\le C(\theta,p)n^{(2/3)p} \quad 0<p<3. 
\end{split}
\end{equation}
The $n$-dependence in \eqref{eqn: SV-moment} is optimal with no $\epsilon$-loss, but only low moments can be controlled.

Theorem \ref{thm: main-moments} is based on an inductive argument involving two inequalities. A crucial tool is an expression for the $k$-th cumulant of $\log Z_{n,t}^\theta$ as a sum of multilinear expressions in expectations of products of quenched cumulants of $s_0^+$, the positive part of $s_0$, as well as lower order powers of $\log Z^\theta_{n,t}$. This relation between the free energy and the first jump in the system leads to a ``scaling relation" which allows us to simultaneously control $s_0^+$ (or $s_0^-$) and $\log Z_{n,t}^\theta$. 

In order to state the expression for the $k$-th cumulant of $\log Z_{n,t}^\theta$, let $H_{n,\sigma^2}(x)$ denote the $n$-th Hermite polynomial with respect to a Gaussian of variance $\sigma^2$, defined in \eqref{eqn: hermite-var}, and $\psi_k(\theta)$ be the $k$-th derivative of the digamma function \eqref{eqn: digamma}. Let $\kappa_k(X)$ denote the $k$-th cumulant of the random variable $X$. The $k$-th cumulant of a function $f$ with respect to the quenched measure in \eqref{eqn: quenched-E} is denoted by $\kappa^\theta_k(f)$. See Section \ref{sec: cumulants} for details.
\begin{theorem}\label{theorem: explicit formula general} For integers $k\geq 2$,
\begin{equation}\label{eqn: explicit formula}
    \begin{split}
&\kappa_k(\log Z_{n,t}^\theta) +n(-1)^{k-1} \psi_{k-1}(\theta)+t\cdot \delta_{k,2}\\
=&\sum_{\pi\in \mathcal{P}} (|\pi|-1)!(-1)^{|\pi|} \sum_{j=1}^{k-1}\binom{k}{j} \prod_{B\in \pi} \mathbb{E}\left[(\ov{\log Z_{n,t}^\theta})^{a_{j,B}}H_{b_{j,B},t}(B_0(t))\right],
\end{split}
\end{equation}
 where $\mathcal{P}$ ranges over partitions $\pi$ of $\{1,\ldots, k\}$, $a_{j,B}=|B \cap\{1,\ldots, j\}|$, $b_{j,B}=|B \cap \{j+1,\ldots,k\}|=|B|-a_{j,B}$, and $\delta_{i,j}$ is the Kronecker delta function.  We can omit any product of blocks that has a block $B$ completely contained inside $\{j+1,\ldots,k\}$, as well as any partition that contains a singleton.  

Moreover, each factor in the products appearing in \eqref{eqn: explicit formula} has an expression in terms of quenched cumulants of $s_0^+$:
\begin{equation*}
\mathbb{E}\left[(\ov{\log Z_{n,t}^\theta})^{a}H_{b,t}(B_0(t))\right]=(-1)^{b}\sum_{\substack{\ell_1+\cdots+\ell_{a}=b\\ \ell_i\ge 0}}\frac{b!}{\ell_1!\cdots \ell_{a}!}\mathbb{E}\left[\prod_{i=1}^{a} \kappa_{\ell_i}^\theta(s_0^+)\right],
\end{equation*}
where we use the convention $\kappa_0^\theta(s_0^+):=\ov{\log Z_{n,t}^\theta}$.
\end{theorem}

The case $k=2$ was previously obtained by Sepp\"al\"ainen and Valk\'o in \cite{SV}.  For explicit expressions when $k=3$ or $k=4$, see Lemma \eqref{lemma: 3rd cumulant explicit} and Corollary \eqref{cor: explicit 4} respectively.   

\subsection{Outline of paper}
In Section \ref{sec: basics}, we introduce some basic definitions, and review elementary properties of the stationary polymer which appeared in previous literature. We also introduce the notation used throughout the paper.

In Section \ref{sec: IBP}, we use the Cameron-Martin-Girsanov theorem to derive formulas of ``integration by parts" type, relating the positive part of the first jump, $s_0^+$, to the free energy, $\log Z_{n,t}^\theta$, by perturbing the path $B_0(t)$, $t\ge 0$. These formulas are generalizations of a relation in \cite{SV}, which was used to derive the variance estimate 
\begin{equation}
cn^{2/3} \le \Var(\log Z_{n,t}^\theta)\le Cn^{2/3}\label{eq: var estimate}
\end{equation}
for some $n$-independent constants $c$, $C>0$.

Section \ref{sec: convexity} serves as an illustration of the general methodology used to derive Theorem \ref{thm: main-moments}, exploiting the reciprocal relation between $s_0^+$ and $\log Z_{n,t}^\theta$. Using convexity of the free energy of the stationary polymer, we give an alternate, shorter proof of the upper bound of the variance estimate \eqref{eq: var estimate}, first obtained in \cite{SV}.

In Section \ref{sec: formulas}, we exploit Gaussian integration by parts to derive a formula for the cumulants of $\log Z_{n,t}^\theta$ in terms of multilinear expressions in expectations of lower moments of $\log Z_{n,t}^\theta$ and quenched cumulants of $s_0^+$. The formula, which appears in Theorem \ref{theorem: explicit formula general}, is a generalization of the variance identity in \cite{SV}, and it facilitates an inductive analysis of the moments of $\log Z_{n.t}^\theta$: higher central moments of the free energy are estimated by lower moments, as well as lower moments of $s_0^+$.

In Section \ref{sec: estimates}, we use the formula in Theorem \ref{theorem: explicit formula general} to obtain near-optimal bounds on the central moments of the free energy of the stationary polymer, as well as annealed moments of the first jump in the system. Our proof is iterative, combining two inequalities to improve bounds on $\log Z^\theta_{n,t}$ using estimates on the tail of $s_0^+$, and vice versa, with a ``fixed point" at the optimal values of the exponents $(\chi,\xi)=(1/3,2/3)$. An important observation here is that a high probability bound of the form $s_0^+\ll \tau$ implies that $\log Z_{n,t}^\theta$ is insensitive to perturbations of the boundary path $B_0(s)$, $0\le s\le t$ that affect it only for $s\gg \tau$.

\paragraph{Acknowledgements} P.S. wishes to thank H.T. Yau for his hospitality at NTU in Taipei, and for discussions about the OY polymer that led to the results presented in this paper. He also thanks Benjamin Landon for discussions about the polymer. P.S.'s work is partially supported by NSF Grant DMS 1811093.  C.N. would like to thank Hans Chaumont for useful feedback.

\section{Preliminaries and notation}\label{sec: basics}
In this paper, we denote by $\mathbb{P}$ and $\mathbb{E}$ the probability measure, resp.\ expectation on the common probability space $\Omega$ where the two-sided Brownian motions $(B_n(t))_{t\in \mathbb{R}}$, $n=0,1,2,\ldots$ are defined. For a random $X$ on $\Omega$, we denote the centered random variable as follows: 
\[\overline{X}:=X-\mathbb{E}[X].\]
The covariance and variance with respect to $\mathbb{E}$ are respectively denoted by
\begin{align*}
\mathbb{C}\mathrm{ov}(X,Y):=\mathbb{E}[XY]-\mathbb{E}[X]\mathbb{E}[Y]\quad \text{ and }\quad 
\mathbb{V}\mathrm{ar}(X):=\mathbb{C}\mathrm{ov}(X,X)=\mathbb{E}[(\ov{X})^2].
\end{align*}

\subsection{Cumulants}\label{sec: cumulants}
The main input for the computations presented in this paper is Proposition \ref{prop: burke}. That result provides explicit formulas for the cumulants of $s_0$, the first jump in the system. To explain this, introduce the gamma function, defined for $\theta >0$ by
\begin{equation}\label{eqn: gamma-def}
\Gamma(\theta)=\int_0^\infty s^{\theta-1} e^{-s}\,\mathrm{d}s.
\end{equation}
The \emph{digamma} function is the logarithmic derivative of $\Gamma$
\begin{equation}\label{eqn: digamma}
\psi_0(s)=\frac{\Gamma'(s)}{\Gamma(s)}.
\end{equation}
The higher derivatives are denoted by $\psi_k$, $k=1,2,\dots$
\[\psi_k(s)=\frac{\mathrm{d}^k}{\mathrm{d}s^k}\psi_0(s).\]
We have $(-1)^k \psi_k(s)<0$ for any $k\in \N$ and $s>0$, \cite{S}.
By taking expectations in equation \eqref{eqn: burke}, we find
\begin{equation}\label{eqn: cgf}
\mathbb{E}[\log Z_{n,t}^\theta]=-n\psi_0(\theta)+\theta t.
\end{equation}
The relation \eqref{eqn: cgf} gives an expression for the expected cumulant generating function of $s_0$, the first jump in the system. 

Recall that for a random variable $X$ with exponential moments, the $k$-th cumulant, denoted by $\kappa_k(X)$, is equal to the $k$-th derivative at zero of the $\log$-moment generating function. To define the \emph{quenched cumulants}, let  $0<\delta<1$ and let $f:\mathbb{R}^n \rightarrow \mathbb{R}$ satisfy
\eqref{eqn: f-condition}. The cumulant generating function of $f$ is given by
\begin{equation}\label{eqn: delta-Z}
 \log Z_{n,t}^{\theta,\delta f}:= \log \int_{-\infty <s_0<\ldots <s_{n-1}<t}e^{\delta f(s_0,\cdots,s_{n-1})-\mathcal{E}_{n,t}^\theta(s_0,\cdots,s_{n-1})}\underline{\mathrm{d}s}.
\end{equation}
The $k$-th \emph{quenched cumulant} with respect to $E_{n,t}^\theta[\cdot ]$ is then
\[\kappa_k^\theta(f):= \frac{\mathrm{d}^k}{\mathrm{d}\delta^k}\log Z_{n,t}^{\theta,\delta f}\Big|_{\delta=0}.\]
For example,
\begin{align*}
    \kappa_1^\theta(f)=E_{n,t}^\theta[f]\quad \text{ and } \quad
    \kappa_2^\theta(f)=E_{n,t}^\theta[f^2]-(E_{n,t}^\theta[f])^2.
\end{align*}
Note that we suppress the dependence on $n$ and $t$ from the notation for simplicity.

Differentiating \eqref{eqn: cgf} with respect to $\theta$, we have
\begin{equation}\label{eqn: psis}
\mathbb{E}[\kappa^\theta_k(s_0)]=t\delta_{k,1}-n\psi_k(\theta).
\end{equation}
Thus, Proposition \ref{prop: burke}  implies that all \emph{expected} quenched cumulants of $s_0$ for $k\ge 2$ are of order $n$. Similarly, Proposition \ref{prop: burke} implies that
 for each $t>0$, $k\ge 1$, and $1\le j\le n$:   
\begin{equation}\label{eqn: r_j-var}
\kappa_{k+1}(r_j^\theta(t))=(-1)^{k+1}\psi_k(\theta).
\end{equation}

\subsection{A priori bounds}
In this section, we collect a few basic bounds on the quantities we will be interested in under the condition \eqref{eqn: char-dir-1}.
For $x,y\in \mathbb{R}$, we denote the minimum and maximum of $x$ and $y$ by
\begin{equation*}
    x\wedge y =\min\{x,y\}\quad \text{ and } \quad x\vee y=\max\{x,y\}.
\end{equation*}
The positive and negative parts of $x$ are denoted by
\begin{align*}
    x^+=\max\{0,x\}\quad \text{ and }\quad 
    x^-=\max\{0,-x\}.
\end{align*}

An immediate consequence of Proposition \ref{prop: burke} is that $\log Z_{n,t}^\theta$ has finite exponential moments.  Moreover, if we define
\begin{equation}\label{eqn: R-def}
R:=\sum_{j=1}^n r_j^\theta(t),
\end{equation}
we see that for $p\geq 1$, the centered free energy $\overline{\log Z_{n,t}^\theta}$ satisfies
\begin{equation}\label{eqn: a-priori-centered}
\begin{split}
\mathbb{E}[|\overline{\log Z^\theta_{n,t}}|^p]^{1/p}\le \mathbb{E}[|B_0(t)|^p]^{1/p}+\mathbb{E}[|\overline{R}|^p]^{1/p}\le C'(\theta,p)(\sqrt{t}+\sqrt{n}).
\end{split}
\end{equation}
From \cite[Lemma 4.4]{SV}, we also have
\begin{equation}\label{eqn: moment-a-priori}
    \mathbb{E}^\theta[|s_0|^p]\le C(\theta,p)n^p \quad \text{ for every } p>0.
\end{equation}
Expressing cumulants in terms of moments, we have
\begin{equation*}|\kappa^\theta_k(s_0^+)|\le C(k)E^{\theta}_{n,t}[(s_0^+)^k].
\end{equation*}
Combining this with \eqref{eqn: moment-a-priori} gives
\[\mathbb{E}[|\kappa_k^\theta(s_0^+)|^p]^{\frac{1}{p}}\le C(k) n^{k}<\infty \quad \text{ for every } k\in \N \text{ and } p\geq 1.\]

\section{Gaussian integration  by parts}\label{sec: IBP}
The Hermite polynomials are defined by the formula
\[H_k(x)=e^{-\frac{x^2}{2}}\frac{\mathrm{d}^k}{\mathrm{d}x^k}e^{\frac{x^2}{2}}, \quad k=0,1,2,\ldots.\]
The polynomials are orthogonal with respect to the standard Gaussian measure $\frac{1}{\sqrt{2\pi}}e^{-\frac{x^2}{2}}$. The Hermite generating function is \cite[Eqn. (1.1)]{N}
\begin{equation}\label{eqn: hermite-gf}
    e^{\lambda x-\frac{\lambda^2}{2}}=\sum_{n=0}^\infty \frac{\lambda^n}{n!}H_n(x).
\end{equation}
For $t>0$, we also define the generalized Hermite polynomials, with variance $t$ by
\begin{equation}\label{eqn: hermite-var}
    H_{k,t}(x):=t^{\frac{k}{2}}H_k\big(\frac{x}{\sqrt{t}}\big).
\end{equation}
Rescaling \eqref{eqn: hermite-gf}, we have
\begin{equation}\label{eqn: hermite-t-gf}
    e^{\lambda x-\frac{\lambda^2t}{2}}=\sum_{n=0}^\infty \frac{\lambda^n}{n!}H_{n,t}(x).
\end{equation}

Recall that the cumulants of $s_0^+$ with respect to the quenched measure $P^\theta_{n,t}$ are given by
\begin{equation}
\kappa_k^\theta(s_0^+)=\frac{\mathrm{d}^k}{\mathrm{d}\delta^k}  \log Z_{n,t}^{\theta,\delta s_0^+}\big|_{\delta=0}\,\,\,\,\,\,\text{ for } k\ge 1.
\end{equation}
For $k=0$, we use the convention:
\begin{equation}\label{eqn: kappa0-def}
    \kappa_0^\theta(s_0^+):=\overline{\log Z_{n,t}^\theta}.
\end{equation}

\begin{lemma}\label{lemma: expectation of log times hermite}
 For $t>0$, $j,k\ge 1$,
\begin{equation}\label{eqn: hermite-ibp}
\mathbb{E}[(\overline{\log Z^\theta_{n,t}})^j H_{k,t}(B_0(t))]=(-1)^k\sum_{\substack{\ell_1+\cdots+\ell_j=k\\ \ell_i\ge 0}}\frac{k!}{\ell_1!\cdots \ell_j!}\mathbb{E}\left[\prod_{i=1}^j \kappa_{\ell_i}^\theta(s_0^+)\right].
\end{equation}
\begin{proof}
Let $0<\delta<\min\{\theta,1\}$. The expectation 
\[\mathbb{E}[(\log Z^{\theta,-\delta s_0^+}_{n,t}-\mathbb{E}[\log Z_{n,t}^\theta])^j]\]
equals
\[\mathbb{E}\big[\big(\log \int_{0<s_0<\ldots< s_{n-1}<t}e^{\theta s_0 -B_0(s_0)-\delta s_0^+ +\mathcal{E}_{n,t}(s_0,\dots,s_{n-1})}\underline{\mathrm{d}s}-\mathbb{E}[\log Z_{n,t}^\theta]\big)^j\big].\]
By the Cameron-Martin-Girsanov Theorem \cite[Proposition 4.1.2]{N}, this equals 
\[\mathbb{E}[e^{\delta B_0(t)-\frac{\delta^2}{2}t} (\overline{\log Z_{n,t}})^j].\]
The exponential factor in the expectation is the generating function of the generalized Hermite polynomials \eqref{eqn: hermite-t-gf} with variance $t$, so   \eqref{eqn: hermite-ibp} follows by repeated differentiation with respect to $\delta$.

To justify the use of differentiation under the expectation, we show the difference quotients are dominated independently of $\delta$. The derivative
\begin{equation}\label{eqn: dk}
\frac{\mathrm{d}^{k}}{\mathrm{d}\delta^{k}}\left(\ov{\log Z_{n.t}^{\theta,-\delta s_0^+}}\right)^j
\end{equation}
is a linear combination of products of the form 
\[\prod_{i=1}^j \kappa^{\theta,-\delta}_{\ell_i}(s_0^+),\]
where $\sum \ell_i = k$, and $\kappa_k^{\theta,-\delta}$ is the $k$-th cumulant with respect to the measure 
\[E_{n,t}^{\theta,-\delta}[\ \cdot ]:=\frac{E_{n,t}^\theta[e^{-\delta s_0^+}\ \cdot\ ]}{E_{n,t}^\theta[e^{-\delta s_0^+}]}.\]

Using the trivial estimate
\[E_{n,t}^{\theta,-\delta}[f]\le e^t E_{n,t}^\theta[f]\]
and expressing the cumulants in terms of moments, we see that this is bounded up to a constant by a sum of terms of the form
\[|E_{n,t}^{\theta}[(s_0^+)^k]||\log Z^{\theta,-\delta}_{n,t}|^b,\]
where $b=\#\{i: \ell_i =0\}$. Since 
\[\log Z_{n,0}^\theta -B_n(t)=\log Z_{n,0}^{\theta-\delta} -B_n(t)\le \log Z_{n,t}^{\theta,-\delta}\le \log Z_{n,t}^\theta,\]
and all moments of $s_0^+$ and $\log Z_{n,t}^\theta$ are finite, we find that the derivative \eqref{eqn: dk} is dominated by an integrable function, so the lemma now follows from the Dominated Convergence Theorem.
\end{proof}
\end{lemma}

The next proposition will lead to a generalization of \eqref{eqn: hermite-ibp}.
\begin{proposition}\label{prop: CMG}
Let $b:[0,t]\rightarrow \mathbb{R}$ be a bounded function, and $F\in L^2(\Omega)$ be a functional, continuous in $C([0,t])$, of the path $B_0(s)$,  $0\le s\le t$. Then
\begin{equation}\label{eqn: IBP}
\begin{split}
&\frac{\mathrm{d}}{\mathrm{d}\delta}\mathbb{E}\big[F\big(B_0(s_0)+\delta\int_0^{s_0} b(s)\mathrm{d}s, 0\le s_0\le t\big)\big]\Big|_{\delta=0}\\
=&~\mathbb{E}[F\big(B_0(s_0), 0\le s_0\le t\big) \int_0^t b(s)\,\mathrm{d}B_0(s)].
\end{split}
\end{equation}
\begin{proof}
Applying the Cameron-Martin-Girsanov theorem gives
\begin{equation}\label{eqn: apply-DCT}
\mathbb{E}\big[F\big(B_0(s_0)+\delta\int_0^{s_0} b(s)\mathrm{d}s\big)\big]=\mathbb{E}\big[e^{\delta\int_0^{s_0} b(s)\mathrm{d}B_0(s)-\frac{\delta^2}{2}\|b\|^2_{L^2([0,t])}} F\big(B_0(s_0)\big)\big].
\end{equation}
By the mean value theorem, for every $0<\delta<1$, there is a $\nu\in (0,\delta)$ such that
\begin{equation}
\begin{split}
&\frac{1}{\delta}\big|e^{\delta\int_0^{s_0} b(s)\mathrm{d}B_0(s)-\frac{\delta^2}{2}\|b\|^2_{L^2([0,1])}}-1\big|\\
\le &~ \left|\int_0^{s_0}b(s)\,\mathrm{d}B_0(s)-\nu \|b\|^2_{L^2([0,t])}\right|e^{\nu \int_0^{s_0} b(s)\mathrm{d}B_0(s)-\frac{\nu^2}{2}\|b\|^2_{L^2([0,1])}}\\
\le &~ \big(\big|\int_0^{s_0}b(s)\,\mathrm{d}B_0(s)\big|+\|b\|^2_{L^2([0,t])}\big)(1+e^{\int_0^{s_0} b(s)\mathrm{d}B_0(s)}).\\
\end{split}
\end{equation}

Since $\int_0^{s_0}b(s)\,\mathrm{d}B_0(s)$ is Gaussian with mean zero and variance $\int^{s_0}_0b^2(s)\,\mathrm{d}s$, it follows that the difference quotients
\[\frac{1}{\delta}(e^{\delta\int_0^{s_0} b(s)\mathrm{d}B_0(s)-\frac{\delta^2}{2}\|b\|^2_{L^2([0,1])}}-1)F(B_0(s_0))\]
are dominated by an integrable function, so the result follows by applying the Dominated Convergence Theorem to the right side of \eqref{eqn: apply-DCT}.
\end{proof}
\end{proposition}

\begin{corollary}\label{cor: exchange}
Let $0<\tau\leq t$, and $j,k\ge 0$. We have
\begin{equation*}
\mathbb{E}[(\overline{\log Z^\theta_{n,t}})^j H_{k,\tau}(B_0(\tau))]=(-1)^k\sum_{\substack{\ell_1+\cdots+\ell_j=k\\ \ell_i\ge 0}}\frac{k!}{\ell_1!\cdots \ell_j!}\mathbb{E}\left[\prod_{i=1}^j \kappa_{\ell_i}^\theta(s_0^+\wedge \tau)\right].
\end{equation*}
\begin{proof}
 Apply Proposition \ref{prop: CMG} with $(\ov{\log Z_{n,t}^\theta})^j$ and $b(s)=\mathbbm{1}_{\{[0,\tau]\}}$, so
 \[\int_0^{s_0}b(s)\mathrm{d}s=s_0^+\wedge \tau.\]
Differentiation inside the expectation is justified as in the proof of Lemma \ref{lemma: expectation of log times hermite}.
\end{proof}
\end{corollary}

\subsection{Application: Sepp\"al\"ainen and Valk\'o's variance identity}
Recall the notation from \eqref{eqn: R-def}: $R=\sum_{j=1}^n r_j^\theta(t)$.  By Proposition \ref{prop: burke},
\[\overline{R}=\overline{\log Z^\theta_{n,t}}+B_0(t).\]
Squaring both sides, taking expectations, and using \eqref{eqn: r_j-var}, we obtain
\begin{equation}\label{eqn: expand}
    \mathbb{E}[(\overline{R})^2]=n\psi_1(\theta)=\Var(\log Z^\theta_{n,t})+t+2\mathbb{E}[\log Z^\theta_{n,t}B_0(t)].
\end{equation}
Applying the integration by parts formula \eqref{eqn: hermite-ibp} with $j=k=1$ (or \eqref{eqn: IBP} with $b(s)=1_{[0,t]}(s)$), we obtain the identity
\begin{equation}
\mathbb{E}[\log Z^\theta_{n,t} B_0(t)]=-\mathbb{E}[E^\theta_{n,t}[s_0^+]].
\end{equation}
Plugging this into \eqref{eqn: expand} and rearranging yields the key variance identity
\begin{equation}\label{eqn: logZvar}
    \Var(\log Z^\theta_{n,t})=n\psi_1(\theta)-t+2\mathbb{E}^\theta_{n,t}[s_0^+].
\end{equation}
Similar identities relating the variance of a free energy to transversal fluctuations have appeared in several works of Sepp\"al\"ainen and collaborators on studying anomalous fluctuations in KPZ models. See \cite[Theorem 3.6]{SV} and \cite[Theorem 3.7]{S}. One of our main results yields higher order versions of \eqref{eqn: logZvar}.

\section{Convexity proof of Sepp\"al\"ainen and Valk\'o's fluctuation estimate}\label{sec: convexity}
In this section, we present an alternative proof of the estimate
\begin{equation}\label{eqn: the-bound}
    \Var(\log Z_{n,t}^\theta)\le C(\theta)n^{2/3} 
\end{equation}
given the following \emph{characteristic direction} condition
\begin{equation}\label{eqn: char-dir}
|t-n\psi_1(\theta)|\le An^{2/3}.
\end{equation}
\eqref{eqn: the-bound} and the corresponding lower bound were originally obtained by Sepp\"al\"ainen and Valk\'o \cite{SV}. We replace the key step in their proof by the convexity of the free energy.
\begin{lemma}
Almost surely, the function
\[\theta \mapsto \log Z_{n,t}^\theta\]
is convex for all $t$. 
The first derivative with respect to $\theta$ equals
\[E^\theta_{n,t}[s_0],\]
while the second derivative with respect to $\theta$ equals
\[\mathrm{Var}^\theta(s_0):=E_{n,t}^\theta[(s_0-E_{n,t}^\theta[s_0])^2]\ge 0.\]
In particular, for $\eta<\theta<\lambda$, almost surely, we have
\begin{equation}\label{eqn: convexity}
\frac{\log Z_{n,t}^\theta-\log Z_{n,t}^\eta }{\theta-\eta}\le E^\theta_{n,t}[s_0]\le \frac{\log Z_{n,t}^\lambda-\log Z_{n,t}^\theta}{\lambda-\theta}.
\end{equation}
\end{lemma}

The following computation relates the quenched second moment and variance of $s_0$, to those of $s_0^+$.  For simplicity, in the rest of this section, we write $E=E^\theta_{n,t}$.
\begin{lemma}\label{lemma: Quenched variance expansion}
Almost surely,
\begin{equation}\label{eqn: variance-bd}
    E[(s_0-E[s_0])^2]=E[(s_0^+-E[s_0^+])^2]+E[(s_0^--E[s_0^-])^2]+2E[s_0^+]E[s_0^-].
\end{equation} 
In particular, 
\begin{equation}\label{eqn: particular}
\begin{split}
    \E[E[s_0^+]^2]&\le \E[E[s_0]^2]+2\E[E[s_0^+]E[s_0^-]]\\
    &\le \E[E[s_0]^2]-n\psi_2(\theta).
    \end{split}
\end{equation}
\begin{proof}
By direct computation,
\begin{align*}
  E[(s_0-E[s_0])^2]&=E[((s_0^+-E[s_0^+])-(s_0^--E[s_0^-]))^2]\\
  &=E[(s_0^+-E[s_0^+])^2]+E[(s_0^--E[s_0^-])^2]\\
  &\quad -2E[(s_0^+-E[s_0^+])(s_0^--E[s_0^-])].
\end{align*}
Since $s_0^+$ an $s_0^-$ have disjoint support,
\[E[(s_0^+-E[s_0^+])(s_0^--E[s_0^-])]=-E[s_0^+]E[s_0^-],\]
which yields \eqref{eqn: variance-bd}.  All terms in \eqref{eqn: variance-bd} are non-negative, so
  \begin{equation*}
    0\le E[s_0^+]E[s_0^-]\le \frac{1}{2}E[(s_0-E[s_0])^2]=\frac{1}{2}\kappa_2^\theta(s_0).
  \end{equation*}
  Taking expectations and using \eqref{eqn: psis},
  \begin{equation}\label{eqn: cross-term-exp}
    \mathbb{E}[ E[s_0^+]E[s_0^-]]\le -\frac{n}{2}\psi_2(\theta).
  \end{equation}
  Finally, after expanding, we get
  \begin{equation*}
    E[s_0]^2=E[s_0^+]^2+E[s_0^-]^2-2E[s_0^+]E[s_0^-].
  \end{equation*}
Taking expectations and applying \eqref{eqn: cross-term-exp} yields \eqref{eqn: particular}.

\end{proof}
\end{lemma}

The following property regarding the map $\theta\mapsto \Var(\log Z_{n,t}^\theta)$ was already used by Sepp\"al\"ainen and Valk\'o. See \cite[Lemma 4.3]{SV}.
\begin{lemma}
For $\theta,\lambda>0$,
\begin{equation}\label{eqn: comparison}
|\Var(\log Z_{n,t}^\lambda)-\Var(\log Z_{n,t}^\theta)|\le n|\psi_1(\lambda)-\psi_1(\theta)|.
\end{equation}
\end{lemma}
\begin{proof}[Proof of estimate \eqref{eqn: the-bound}]
By \eqref{eqn: convexity} with $\lambda-\theta=\theta-\eta=n^{-1/3}$, 
\[n^{-1/3}|E[s_0]|\le |\log Z_{n,t}^\theta-\log Z_n^\lambda|+|\log Z_{n,t}^\eta-\log Z_{n,t}^{\theta}|.\]
Note that 
\[|\psi_0(\theta)-\psi_0(\lambda)-(\lambda-\theta)\psi_1(\theta)|\le C(\lambda-\theta)^2.\]
Combined with \eqref{eqn: cgf} and \eqref{eqn: char-dir}, we can center the free energies to obtain:
\begin{align*}
  &|\log Z_n^\theta-\log Z_n^\lambda|+|\log Z_n^\eta-\log Z_n^\theta|\\
  \le&~ An^{\frac{1}{3}}\big(|\theta-\lambda|+|\eta-\theta|\big)+ Cn\big((\lambda-\theta)^2+(\eta-\theta)^2\big)+|\overline{\log Z_n^\theta}-\overline{\log Z_n^\lambda}|+|\overline{\log Z_n^\eta}-\overline{\log Z_n^\theta}|\\
  \le&~ Cn^{1/3}+|\overline{\log Z_n^\theta}-\overline{\log Z_n^\lambda}|+|\overline{\log Z_n^\eta}-\overline{\log Z_n^\theta}|.
\end{align*}
Squaring, taking expectations, and using \eqref{eqn: comparison}, we have the bound
\begin{equation}\label{eqn: E2}
\begin{split}
n^{-2/3}\mathbb{E}[E[s_0]^2] &\le C(\theta)(n^{2/3}+\mathbb{E}[|\overline{\log Z_n^\theta}-\overline{\log Z_n^\lambda}|^2]+\mathbb{E}[|\overline{\log Z_n^\eta}-\overline{\log Z_n^\theta}|^2])\\
&\le C(\theta)(n^{2/3}+\Var(\log Z_{n,t}^\theta)+n|\lambda-\theta|+n|\eta-\theta|).
\end{split}
\end{equation}

Using \eqref{eqn: particular}, we find
\[\mathbb{E}[E[s_0^+ ]]\le \mathbb{E}[E[s_0]^2]^{1/2}+C(\theta)n^{1/2}.\]
Finally,  \eqref{eqn: logZvar}, \eqref{eqn: char-dir}, and \eqref{eqn: E2} give
\[\Var(\log Z_{n,t}^\theta)\le Cn^{1/3}(n^{2/3}+\Var(\log Z_{n,t}^\theta))^{1/2},\]
a quadratic relation which implies \eqref{eqn: the-bound}.
\end{proof}

\section{Formulas for $\kappa_k(\log Z_{n,t}^\theta)$}\label{sec: formulas} In order to give exact formulas for $\kappa_k(\log Z^\theta_{n,t})$ we first discuss joint cumulants and their connection to Hermite polynomials. The joint cumulant of the random variables $X_1,\dots, X_k$ is defined by
\begin{equation}
\kappa(X_1,\dots,X_k):=\frac{\partial^k}{\partial\xi_1\dots\partial\xi_k}\log \E[e^{\sum_{j=1}^k\xi_j X_j}]\Big|_{\xi_i=0}.\label{eq: joint culumant formula 1}
\end{equation}
Alternatively, it can be written as a combination of products of expectations of the underlying random variables:
\begin{equation}
\kappa(X_1,\dots,X_k)=\sum_{\pi\in\mathcal{P}}(|\pi|-1)!(-1)^{|\pi|-1}\prod_{B\in \pi}\E\left[\prod_{i\in B} X_i \right]\label{eq: joint cumulant formula 2}
\end{equation}
where $\mathcal{P}$ ranges over partitions $\pi$ of $\{1,\ldots, k\}$ and $|A|$ stands for the size of the set $A$.  Note that the joint cumulant is multilinear.  In the case where $X_1=X_2=\cdots =X_k=X$, the joint cumulant reduces to the $k$-th cumulant of $X$, $\kappa_k(X)$.  Two important properties of cumulants that we will take advantage of are shift-invariance: 
\[
\kappa_k(X+c)=\kappa_k(X) \text{ for } k\geq 2, \text{ where } c \text{ is constant,}
\]
and additivity for independent random variables:
\[
\kappa_k(X+Y)=\kappa_k(X)+\kappa_k(Y) \text{ for any } k, \text{ if } X \text{ and } Y \text { are independent}.
\]
The following lemma relates the $k$-th cumulant of the free energy to a sum of joint cumulants involving the centered free energy Brownian motion $B_0$.

\begin{lemma}\label{Lemma-cumulant expression}
Let $\theta>0$, $t>0$, 
 and $n\in\N$.   Then for any integer $k\geq 2$,
\begin{equation}
\kappa_k(\log Z_{n,t}^\theta)= n(-1)^k \psi_{k-1}(\theta)-\sum_{j=0}^{k-1}\binom{k}{j}\kappa(\underbrace{\ov{\log Z_{n,t}^\theta},\ldots,\ov{\log Z_{n,t}^\theta}}_{j\text{-times}},\underbrace{B_0(t),\dots,B_0(t)}_{k-j \text{-times}}).\label{eq: 1st lemma-cumulant expression}
\end{equation}
Note that the $j=0$-th term in the summation is $\kappa_k(B_0(t))$ which equals $0$ when $k\neq 2$, and $t$ when $k=2$.
\end{lemma}

\begin{proof}
For convenience, put $A:=\ov{\log Z_{n,t}^\theta}$, $B_0:=B_0(t)$, and $R:=\sum_{j=1}^n r_j^\theta(t)$, so $\overline{R}=A+B_0$.  The shift-invariance of the cumulant along with the mulitlinearity of the joint cumulant gives

\[
\kappa_k(R)=\kappa_k(\overline{R})=\kappa(\underbrace{A+B_0,A+B_0,\dots,A+B_0}_{k\text{-times}})=\sum_{i=0}^k \binom{k}{i}\kappa(\underbrace{A,\dots,A}_{j\text{-times}},\underbrace{B_0,\dots,B_0}_{k-j \text{-times}}).
\]
The left-hand side simplifies to $\kappa_k(R)=n\kappa_k(r_j^\theta(t))=n(-1)^k \psi_{k-1}(\theta)$ by equation $\eqref{eqn: r_j-var}$,  as $R$ is a sum of $n$ i.i.d.\ random variables, while the $k$-th entry in the sum on the right-hand side gives $\kappa_k(\log Z^\theta_{n,t})$.  Rearranging yields the desired result.
\end{proof}
 
 \subsection{ Estimate for $\kappa_3(\log Z_{n,t}^\theta)$}
To motivate  computations  in the upcoming sections we use Lemma \ref{Lemma-cumulant expression} and \cite[Eqn. (4.13)]{SV} to obtain a bound of the optimal order, $n^{(1/3)\cdot3}$
, for the third centered moment of $\log Z_{n,t}^\theta$. 

The joint cumulants simplify when the random variables are centered.  For example, if $X,Y,Z$ are centered, then
\begin{equation}
\kappa(X,Y,Z)=\E[XYZ].\label{eq: 3rd joint cumulant}    
\end{equation}
Therefore the third cumulant of a random variable agrees with its third central moment.   
We now use \eqref{eq: 3rd joint cumulant}  to obtain an exact formula for the third cumulant/central moment of the free energy.
\begin{lemma}\label{lemma: 3rd cumulant explicit}  For any $t>0$ and $n\in\N$,
\begin{equation}
\E[(\ov{\log Z_{n,t}^\theta})^3]=\kappa_3(\log Z_{n,t}^\theta)= - n\psi_2(\theta) +6 \E[\overline{\log Z_{n,t}^\theta} E_{n,t}^\theta[s_0^+]]-3\E[\mathrm{Var}^{\theta}(s_0^+)].
\end{equation}
\end{lemma}
\begin{proof}
For convenience we write $Z=Z_{n,t}^\theta$, $B_0=B_0(t)$, and $E=E^\theta_{n,t}$.
By Lemma \ref{Lemma-cumulant expression}, 
\begin{equation}
\kappa_3(\log Z)=-n\psi_2(\theta)-3\kappa(\ov{\log Z},\ov{\log Z},B_0)-3\kappa(\ov{\log Z},B_0,B_0) \label{eq: 3rd moment -1}
\end{equation}
We now analyze the joint cumulants individually. Equation \eqref{eq: 3rd joint cumulant} and two applications of Lemma \ref{lemma: expectation of log times hermite} give

\begin{equation}
\kappa(\overline{\log Z},\overline{\log Z},B_0)=\E[\overline{\log Z}^2B_0]=-2\E[\overline{\log Z}E[s_0^+]],\label{eq: 3rd moment -2}
\end{equation}
and
\begin{equation}
\kappa(\overline{\log Z},B_0,B_0)=\E[\overline{\log Z}B_0^2]=\E[\log Z (B_0^2-t)]=\E[\textrm{Var}^\theta(s_0^+)].\label{eq: 3rd moment -3}
\end{equation}
Combining equations \eqref{eq: 3rd moment -1}, \eqref{eq: 3rd moment -2}, and \eqref{eq: 3rd moment -3} yields the desired result.
\end{proof}

Next, we use Lemma \ref{lemma: 3rd cumulant explicit} to show that $\kappa_3(\log Z^\theta_{n,t})$ has order at most $n$ when $n$ and $t$ satisfy \eqref{eqn: char-dir-1}.

\begin{corollary}
Assume $n$ and $t$ satisfy
\[
|t-n\psi_1(\theta)|\leq A n^{2/3}.
\]
Then there exists a constant $C=C(\theta)<\infty$ such that for all $n\in \N$
\[
\left|\E[(\ov{\log Z_{n,t}^\theta})^3]\right|\leq C n.
\]
\end{corollary}

\begin{proof}
Applying the Cauchy-Schwartz inequality followed by Jensen's inequality, \cite[Eqn. (4.13)]{SV}, and the bound \eqref{eqn: the-bound},
\begin{align*}
\left|\E\left[\overline{\log Z_{n,t}^\theta}E^{\theta}_{n,t}[s_0^+]\right]\right|&\leq \E\left[\left(\overline{\log Z_{n,t}^\theta}\right)^2\right]^{\frac{1}{2}}\E\left[E^{\theta}_{n,t}\left[(s_0^+)^2\right]\right]^{\frac{1}{2}}\\
&\leq C(n^{\frac{2}{3}})^\frac{1}{2}(n^{\frac{4}{3}})^\frac{1}{2}=C n.
\end{align*}

By equation \eqref{eqn: psis}, we have 
\[
0\leq\E[\mathrm{Var}^\theta(s_0^+)]\leq \E[\mathrm{Var}^\theta(s_0)] = -n\psi_2(\theta).
\]
\end{proof}

\subsection{Higher cumulants: Proof of Theorem \ref{theorem: explicit formula general}}
We now develop a systematic method to deal with higher cumulants.    The following lemma expresses the joint cumulants appearing in the sum on the right-hand side of equation \eqref{eq: 1st lemma-cumulant expression} as linear combinations of products of expectations which only involve the free energy and Hermite polynomials of the Brownian motion $B_0$.  After multiple Gaussian integration by parts, the remaining expressions will involve expectations of quenched cumulants rather than the Brownian motion $B_0$, leading to the exact formula in Theorem \ref{theorem: explicit formula general}.  

\begin{lemma}\label{lemma: joint cumulants to hermite} Let $\theta>0$, $t>0$, $n\in \N$, $k\in \N$, and $1\leq j\leq k$. Then   

\begin{align*}
\kappa(&\underbrace{\ov{\log Z_{n,t}^\theta},\dots,\ov{\log Z_{n,t}^\theta}}_{j\text{-times}},\underbrace{B_0(t),\dots,B_0(t)}_{k-j \text{-times}})\\
 &=
\sum_{\pi\in \mathcal{P}} (|\pi|-1)!(-1)^{|\pi| -1} \prod_{B\in \pi} \mathbb{E}\left[(\ov{\log Z_{n,t}^\theta})^{|B \cap\{1,\ldots, j\}|}H_{|B \cap \{j+1,\ldots,k\}|,t}(B_0(t))\right],
\end{align*}
where $\mathcal{P}$ ranges over partitions $\pi$ of $\{1,\ldots, k\}$.  We can omit any partition $\pi$ which has a block $B$ contained in $\{j+1,\ldots, k\}$.  We can also omit any partition $\pi$ which contains a singleton set.
\end{lemma}

\begin{proof}
For convenience, again put $A=\ov{\log Z_{n,t}^\theta}$ and $B_0=B_0(t)$. Recalling the generalized Hermite generating function \eqref{eqn: hermite-t-gf}, we have
\[e^{\lambda B_0}= e^{\frac{\lambda^2 t}{2}}\sum_{n=0}^\infty \frac{\lambda^n}{n!} H_{n,t}(B_0).\]
Therefore, 
\begin{align*}
&\log \E\left[e^{(\xi_1+\dots + \xi_j)A+(\xi_{j+1}+\dots + \xi_k)B_0}\right]\\
&=\log \mathbb{E}\left[\sum_{n=0}^\infty e^{(\xi_1+\cdots+\xi_j)A}(\xi_{j+1}+\cdots+\xi_k)^n H_{n,t}(B_0)\right]+\frac{(\xi_{j+1}+\cdots+\xi_k)^2t}{2}.
\end{align*}
Plugging this into the right-hand side of \eqref{eq: joint culumant formula 1}, taking the derivatives $\partial_{\xi_1},\cdots,\partial_{\xi_k}$, evaluating at $\xi_i=0$, and using $\mathbb{E}[H_{n,t}(B_0)]=0$ for $n\ge 1$, we obtain the formula
\[
\kappa(\underbrace{A,\dots,A}_{j\text{ times}},\underbrace{B_0,\dots,B_0}_{k-j \text{ times}})=\sum_{\pi\in \mathcal{P}} (|\pi|-1)!(-1)^{|\pi| -1} \prod_{B\in \pi} \mathbb{E}\left[A^{|B \cap\{1,\ldots, j\}|}H_{|B \cap \{j+1,\ldots,k\}|,t}(B_0)\right]\]
where $\mathcal{P}$ ranges over partitions $\pi$ of $\{1,\ldots, k\}$ such that no block $B\in \pi$ is contained in $\{j+1,\ldots, k\}$.  Finally, if $B$ is a singleton set that is contained in $\{1,\dots,j\}$, then
\[
\mathbb{E}\left[A^{|B \cap\{1,\ldots, j\}|}H_{|B \cap \{j+1,\ldots,k\}|,t}(B_0)\right]=\E[A]=0.
\]
\end{proof}

We can now prove Theorem \ref{theorem: explicit formula general}.

\begin{proof}[Proof of Theorem \ref{theorem: explicit formula general}]
Combine Lemmas \ref{Lemma-cumulant expression}, \ref{lemma: joint cumulants to hermite}, and \ref{lemma: expectation of log times hermite}.
\end{proof}

One can verify that the formula for $k=3$ agrees with that in Lemma \ref{lemma: 3rd cumulant explicit}.
For another concrete exact formula, one can verify that the formula for $k=4$ gives
\begin{corollary}\label{cor: explicit 4}
\begin{align*}
\kappa_4(\log Z_{n,t}^\theta)=~& n\psi_3(\theta)+4\mathbb{E}[\kappa_3^\theta(s_0^+)]+12\Cov(E^\theta_{n,t}[s_0^+],(\ov{\log Z_{n,t}^\theta})^2)\\
&-12\Var(E[s_0^+])-12\mathbb{E}[\mathrm{Var}^\theta(s_0^+)\ov{\log Z_{n,t}^\theta}].
\end{align*}
\end{corollary}

\section{Estimates for the central moments: Proof of Theorem \ref{thm: main-moments}}\label{sec: estimates}
The proof of Theorem \ref{thm: main-moments} is obtained by iterating the two inequalities \eqref{eqn: T-bound} and \eqref{eqn: est}. These relate the moments of $s_0^+$ and the central moments of $\log Z^\theta_{n,t}$, successively improving bounds for both. The inequality \eqref{eqn: T-bound} exploits the relationship between $n$ and $t$ given in \eqref{eqn: char-dir-1} to obtain a first order cancellation, see \cite[Lemma 4.2]{SV}. The case $k=2$ was used by the authors of \cite{SV} to estimate the variance of the partition function, and similar bounds appear in works of Sepp\"al\"ainen \cite{S} and Bal\'azs-Cator-Sepp\"al\"ainen \cite{BCS}. The estimate \eqref{eqn: est} is enabled by the expression in Theorem \ref{theorem: explicit formula general}.

The two inequalities can be interpreted as manifestations of the conjectural scaling relations between the fluctuation exponent $\chi$ and the transversal fluctuation exponent $\xi$ for models in the Kardar-Parisi-Zhang class \cite{KS}:
\begin{equation}
    2\xi\le 1+\chi
\end{equation}
(for \eqref{eqn: T-bound}) and
\begin{equation}
    2\chi \le \xi
\end{equation}
for \eqref{eqn: est}. When combined, these give the bounds
\begin{equation}
    \chi \le \frac{1}{3}, \quad \xi \le \frac{2}{3}.
\end{equation}

We give a brief sketch of the argument for the reader's convenience.
\begin{enumerate}
    \item Assuming the existence of constants  $C,\delta>0$ such that for all $\theta \in [1,L]$, $k\ge 1$
    \begin{equation}
    \mathbb{E}[(\overline{\log Z^{\theta}_{n,t}})^k]\le C(k)n^{(1/3+\delta) k},
    \end{equation}
    we show in Section \ref{sec: tail-bound} the estimate
    \begin{equation}\label{eqn: s0mom}
        \mathbb{E}^{\theta}[(s_0^+)^{\mathbf{2}k}]\le C'(k) n^{(4/3+\delta)k +\epsilon}.
    \end{equation}
    for $\theta\in [1,L-1]$ and some $n$-independent constants $C'(k)$.
    This bound corresponds to the scaling inequality $2\xi\le 1+\chi$.  
    \item  Using Theorem \ref{theorem: explicit formula general}, we have an expression for the cumulants of $\log Z_{n,t}^\theta$ of the following form
    \begin{equation}\label{eqn: cumulanthalves}
     \kappa_k(\log Z^\theta_{n,t})=\sum_{j=1}^{k-1} c_{k,j}\prod_{i\in I_j} \mathbb{E}[(\overline{\log Z_{n,t}^\theta})^{\alpha_{j,i}}  H_{\beta_{j,i},t}(B_0(t))],
    \end{equation}
    $\sum_{i\in I_j} \alpha_{j,i}+\beta_{j,i}\le k$ and $\alpha_{j.i}\le j$.
    \item Time truncation argument: by Corollary \ref{cor: exchange}, we can replace $H_{\beta_{j,i},t}(B_0(t))$ by the smaller quantity $H_{\beta_{j,i},\tau}(B_0(\tau))$ provided
    $s_0^+ \ll \tau$. 
    
    Using \eqref{eqn: s0mom}, we have the truncation:
    \begin{align*}
        \mathbb{E}[(s_0^+)^m, s_0^+> n^{2/3+\delta/2+\epsilon}]&\le n^{-(2k-m)(2/3+\delta/\mathbf{2}+\epsilon)}\mathbb{E}[(s_0^+)^{2k}]\\
        &\le 2k\cdot C(k) n^{(2/3)m+(\delta/\mathbf{2}) m-(2k-m+1)\epsilon}.
    \end{align*}
    This is of sub-leading order if we choose $k\gg (m\delta)/\epsilon$.  
    \item Thanks to the previous truncations, we can now estimate \eqref{eqn: cumulanthalves} by effectively replacing $B_0(t)$ by $B_0(\tau)$, where $\tau\gg s_0^+$ is the best current bound for the typical size of $s_0^+$.  Similarly,  we can replace $H_{k,t}(B_0(t))$ by $H_{k,\tau}(B_0(\tau))$. The moments of the centered free energy $\E[(\ov{\log Z_{n,t}^\theta})^k]$ can now be estimated inductively using \eqref{eqn: cumulanthalves} and
     \begin{equation}
        B_0(\tau) \lesssim \tau^{1/2}.
    \end{equation}
    The last relation plays the role of the scaling inequality $2\chi\leq \xi$.
    
\end{enumerate}

\subsection{Tail bound for $s_0^+$}\label{sec: tail-bound}
The following is one of the two pivotal inequalities in our proof. As previously stated, the case $k=2$ appears in \cite{SV}. See also \cite[Lemma 2.2]{MSV}.
\begin{lemma}\label{lem: SV tail bound}
Let $k\geq 2$ be an even integer, $0<\theta\leq L$, and suppose 
\begin{equation}
    |t-n\psi_1(\theta)|\le An^{2/3}.
\end{equation}
Then there exist constants $s,c,C,K>0$, which are uniformly bounded in $\theta$, such that, if   
\[
n^{2/3} \le u \le Kn \quad \text{ and } \quad  \lambda-\theta=c\frac{u}{n}=\theta-\eta,
\] then the following inequalities hold:
 \begin{align}
 \mathbb{P}(P^\theta_{n,t}(s_0^+> u)\ge e^{-su^2/n})\le C\frac{n^k}{u^{2k}}\big(\mathbb{E}[(\overline{\log Z_{n,t}^\theta})^k]+\mathbb{E}[(\overline{\log Z_{n,t}^\lambda})^k]\big)\label{eqn: T-bound}\\
\mathbb{P}(P^\theta_{n,t}(s_0^-> u)\ge e^{-su^2/n})\le C\frac{n^k}{u^{2k}}\big(\mathbb{E}[(\overline{\log Z_{n,t}^\theta})^k]+\mathbb{E}[(\overline{\log Z_{n,t}^\eta})^k]\big)\label{eqn: T-bound 2}
\end{align}
\end{lemma}
\begin{proof} We first prove \eqref{eqn: T-bound}.
Let $r,u>0$. By Markov's inequality,
\begin{align*}
    P^\theta_{n,t}(s_0^+>u)&= P^\theta_{n,t}(s_0>u)\le e^{-ru}E^\theta_{n,t}[e^{rs_0}]=e^{-ru}\frac{ Z_{n,t}^{\theta+r}}{ Z_{n,t}^\theta}.
\end{align*}
Thus, for any $\alpha>0$,
\begin{align*}
    \mathbb{P}(P^\theta_{n,t}(s_0^+>u)\ge e^{-\alpha})
    \le~&\mathbb{P}(\frac{ Z_{n,t}^{\theta+r}}{ Z_{n,t}^\theta} \ge e^{ru-\alpha})\\
    =~& \mathbb{P}(\log Z_{n,t}^{\theta+r}-\log Z_{n,t}^\theta \ge ru-\alpha)\\
    =~&\mathbb{P}(\ov{\log Z_{n,t}^{\theta+r}}-\ov{\log Z_{n,t}^\theta}\ge n(\psi_0(\theta+r)-\psi_0(\theta))-rt+ru-\alpha).
    \end{align*}
    The last equality follows from \eqref{eqn: cgf}.  For  $c_0=c_0(\theta)$ small enough and $0<r<c_0$,
    \[|\psi_0(\theta+r)-\psi_0(\theta)-r\psi_1(\theta)|\le -2r^2\psi_2(\theta).\]
    Since 
    \[|t-n\psi_1(\theta)|\le An^{2/3},\]
    \begin{equation}
    n(\psi_0(\theta+r)-\psi_0(\theta))-rt+ru-\alpha\ge ru-\alpha-rAn^{2/3}+2r^2\psi_2(\theta).\label{eqn: square lb}
    \end{equation}
    Letting 
    \[r=\lambda-\theta=c\frac{u}{n}   \quad \text{ and }\quad \alpha=\frac{s u^2}{n}\]
    we can ensure the right-hand side of \eqref{eqn: square lb} is at-least $ \frac{cu^2}{2n}$ by first fixing $c$ small enough (depending on $\theta$ and $A$) and  then fixing $s,K$ small enough in relation to $c$.  Finally, apply Markov's inequality using the $k$-th moment.

    To prove \eqref{eqn: T-bound 2}, let $r<0$ and $u>0$.  Then
    \[
    P^\theta_{n,t}(s_0^->u)=P^\theta_{n,t}(s_0<-u)\leq e^{-ru}E^\theta_{n,t}[e^{rs_0}].
    \]
    The rest of the argument is the same as in the previous case.
\end{proof}

\begin{corollary}\label{cor: log Z to s}
Suppose 
\begin{equation*}
    |t-n\psi_1(\theta_0)|\le An^{2/3}.
\end{equation*}
Let $k\ge 2$ be an even integer and suppose that
\begin{equation}\label{eqn: cor assumption}
    \mathbb{E}[(\overline{\log Z_{n,t}^\theta})^k]\le C(k)n^{(1/3)k+\delta k}
\end{equation}
for some $\delta>0$ and all $\theta \in [\theta_0,\theta_0+L]$. 

Then, for any $\epsilon>0$, there exists a constant $C(\epsilon,k,L,\theta_0)$ such that
\begin{equation*}
\mathbb{E}^\theta[(s_0^{\pm})^{2k}]\le C(\epsilon,k,L,\theta_0)n^{(4/3)k+\delta k+\epsilon},
\end{equation*}
for all $\theta\in [\theta_0,\theta_0+L-1]$.
\end{corollary}
\begin{proof}
Write
\begin{align*}
    \mathbb{E}^\theta [(s_0^{\pm})^{2k}]&\le (n^{2/3})^{2k} +(2k)(Kn)^\epsilon \int_{n^{2/3}}^{Kn} u^{2k-1-\epsilon} \mathbb{P}^{\theta}(s_0^{\pm}\ge u)\mathrm{d}u+C(\theta,k)\\
    &= Cn^{(4/3)k+\delta k+\epsilon} \int_{n^{2/3}}^{Kn} u^{-1-\epsilon}\mathrm{d}u+  O(n^{(4/3)k}).
\end{align*}
In the second step, we have used Lemma \ref{lem: SV tail bound} and the assumption \eqref{eqn: cor assumption}. To control the region $\{u\ge c n\}$, we have applied Lemma \cite[Lemma 4.4]{SV}. Performing the integration, we obtain the result.
\end{proof}

\begin{remark}
The reduction in the upper bound on $\theta$ from $\theta_0+L$ to $\theta_0+L-1$ in the conclusion of Corollary \ref{cor: log Z to s} is due to our application of Lemma \ref{lem: SV tail bound}, which requires that the  assumption \eqref{eqn: cor assumption} hold for $\theta$ and $\lambda$, where $\theta< \lambda \ll \theta+ 1$.
\end{remark}

\subsection{Truncation}
\begin{lemma}\label{lem: truncation}
Suppose that $t=O(n)$ and there exist constants $C_k=C(k,\theta)$ for $k\in \N$, which are locally bounded in $\theta$, such that for some $0<\epsilon<\delta/10$ and all $k\in \N$,
\begin{equation}\label{eqn: truncation assumption}
    \mathbb{E}^{\theta}[(s_0^+)^{2k}]\le C(k,\theta)n^{(4/3)k+\delta k+\epsilon} \quad \text{ for all } n\in \N.
\end{equation}
Then, there exist constants $C(j,l,\theta,\epsilon,\delta,K)$ (locally bounded in $\theta$) such that if $j,\ell,K\ge 1$,
\begin{equation*}
    \big|\mathbb{E}[(\overline{\log Z_{n,t}^\theta})^j H_{\ell,\tau}(B_0(\tau))]-\mathbb{E}[(\overline{\log Z_{n,t}^\theta})^j H_{\ell,t}(B_0(t))]\big|\le C(j,l,\theta,\epsilon,\delta, K)n^{-K}\quad \text{ for all } n\in \N,
\end{equation*}
 where 
\begin{equation*}
    \tau=n^{2/3+\delta/2+\epsilon}.
\end{equation*}
\end{lemma}

\begin{remark}
We only require \eqref{eqn: truncation assumption} hold for $s_0^+$.  We could equivalently replace $s_0^+$ with $s_0^-$ in the assumption.
\end{remark}
\begin{proof}
By Corollary \ref{cor: exchange}, for $0\leq\tau\le t$
\[\mathbb{E}[(\overline{\log Z^\theta_{n,t}})^j H_{k,\tau}(B_0(\tau))]=(-1)^k\sum_{\substack{\ell_1+\cdots+\ell_j=k\\ \ell_i\ge 0}}\frac{k!}{\ell_1!\cdots \ell_j!}\mathbb{E}\left[\prod_{i=1}^j \kappa_{\ell_i}^\theta(s_0^+\wedge \tau)\right],\]
where we interpret $\kappa_0^\theta(s_0^+\wedge \tau)=\overline{\log Z_{n,t}^\theta}$.  It will therefore suffice to compare expectations of products of quenched cumulants of $s_0^+$ and $s_0^+\wedge \tau$. Let $I=\{1,\ldots, j\}$. We want to estimate
\begin{equation}
    \label{eqn: difference}
    \mathbb{E}[\prod_{i\in I} \kappa_{\ell_i}(s_0^+)]-\mathbb{E}[\prod_{i\in I} \kappa_{\ell_i}(s_0^+\wedge \tau)],
\end{equation}
where $\sum_i \ell_i=k$.  By a telescoping argument, it is enough to estimate
\[\mathbb{E}[\kappa_{\ell_a}(s_0^+)\prod_{i\in I_1} \kappa_{\ell_i}(s_0^+) \prod_{i\in I_2}\kappa_{\ell_i}(s_0^+\wedge \tau)]-\mathbb{E}[\kappa_{\ell_a}(s_0^+\wedge \tau)\prod_{i\in I_1} \kappa_{\ell_i}(s_0^+) \prod_{i\in I_2}\kappa_{\ell_i}(s_0^+\wedge \tau)],\]
where  $I=I_1\cup I_2 \cup\{a\}$ and $\ell_a\neq 0$ (if $\ell_a=0$, then the difference is zero).  By H\"older's estimate, this difference is bounded by
\begin{equation}\label{eqn: L2}
    \mathbb{E}[|\kappa_{\ell_a}(s_0^+)-\kappa_{\ell_a}(s_0^+\wedge \tau)|^2]^{1/2} \prod_{i\in I_1}\mathbb{E}[|\kappa_{\ell_i}(s_0^+)|^{2j-2}]^{1/(2j-2)}\prod_{i\in I_2}\mathbb{E}[|\kappa_{\ell_i}(s_0^+\wedge \tau)|^{2j-2}]^{1/(2j-2)}.
\end{equation}
To bound the two products in \eqref{eqn: L2}, we use equation \eqref{eq: joint cumulant formula 2} to obtain the estimate
\[|\kappa_{\ell_i}(s_0^+\wedge \tau)|\vee|\kappa_{\ell_i}(s_0^+)|\le (\ell_i-1)!\sum_\pi \prod_{B\in \pi}E[(s_0^+)^{|B|}], \quad \ell_i\neq 0\]
where $\pi$ runs over all partitions of $\{1,\dots,\ell_i\}$ and $E=E^\theta_{n,t}$.
Taking the $L^b$-norm, $b\ge 1$ we have   
\begin{equation}\label{eqn: Lb}
\mathbb{E}[|\kappa_{\ell_i}(s_0^+\wedge\tau)|^b]^{1/b}\vee\mathbb{E}[|\kappa_{\ell_i}(s_0^+)|^b]^{1/b}\le \begin{cases} 
C\sqrt{n}, & \quad \ell_i=0\\
C^{\ell_i}(\ell_i-1)!\mathbb{E}^{\theta}[(s_0^+)^{\ell_i b}]^{1/b},& \quad \ell_i\neq 0.
\end{cases}
\end{equation}
Recall from \eqref{eqn: kappa0-def} that $\kappa_0(s_0^+\wedge \tau)=\kappa_0(s_0^+)=\overline{\log Z_{n,t}^\theta}$, so the case $\ell_i=0$ follows from \eqref{eqn: a-priori-centered} and the fact that $t=O(n)$.
Now define
\[M(I_1,I_2,n):= \{i\in I_1\cup I_2: \ell_i=0\} \quad \text{ and } \quad m_0:= |M(I_1,I_2,n)|.\]
Proceeding with the estimate \eqref{eqn: Lb}, we have
\begin{align}
    &\prod_{i\in I_1}\mathbb{E}[|\kappa_{\ell_i}(s_0^+)|^{2j-2}]^{1/(2j-2)}\prod_{i\in I_2}\mathbb{E}[|\kappa_{\ell_i}(s_0^+\wedge \tau)|^{2j-2}]^{1/(2j-2)}\nonumber\\
    \le&~(C\sqrt{n})^{m_0}\cdot \prod_{i\in I_1,I_2: \ell_i \neq 0 }C^{\ell_i}  (\ell_i-1)!  \mathbb{E}^{\theta}[(s_0^+)^{2j-2}]^{\ell_i/(2j-2)}  \nonumber\\
    \le&~ C^{m_0+k-\ell_a}(k-\ell_a)! n^{m_0/2}   \mathbb{E}^{\theta}[(s_0^+)^{2j-2}]^{(k-\ell_a)/(2j-2)}.\nonumber\\
    \le&~C^{j+k}k! C(j-1,\theta)^{\frac{k-\ell_a}{2j-2}} n^{(j+k-\ell_a)(\frac{2}{3}+\frac{\delta}{2}+\frac{\epsilon}{2})}. \label{eqn: controlling I1I2}
\end{align}
The last inequality follows from  $n^{m_0/2}\le n^{j(2/3+\delta/2+\epsilon/2)}$ and the assumption \eqref{eqn: truncation assumption}.

We now estimate the first factor in \eqref{eqn: L2}. Expressing $\kappa_{\ell_a}(s_0^+)$, $\kappa_{\ell_a}(s_0^+\wedge \tau)$ in terms of moments, we see that it suffices to bound the $L^1$-norm of the difference
\[\prod_{i=1}^r E[(s_0^+\wedge\tau)^{\alpha_i}]-\prod_{i=1}^r E[(s_0^+)^{\alpha_i}],\]
where $r\le \ell_a$ and $\sum \alpha_i = \ell_a$.
By another telescoping argument, it suffices to bound the expectation of
\begin{equation}\label{eqn: applying M}
    \begin{split}
&E[(s_0^+)^{\alpha_v} s_0>\tau]\prod_{i=1}^{v-1}E[(s_0^+)^{\alpha_i}]\prod_{i=v+1}^r E[(s_0^+\wedge \tau)^{\alpha_i}]\\
\le~& \tau^{-M+\alpha_v}E[(s_0^+)^M]\prod_{i=1}^{v-1}E[(s_0^+)^{\alpha_i}]\prod_{i=v+1}^r E[(s_0^+\wedge \tau)^{\alpha_i}]  
\end{split}
\end{equation}
where $M\geq \ell_a$.  
Applying H\"older's inequality to \eqref{eqn: applying M} followed by  \eqref{eqn: Lb} and assumption \eqref{eqn: truncation assumption},
  \begin{align}
      &\mathbb{E}[|\kappa_{\ell_a}(s_0^+)-\kappa_{\ell_a}(s_0^+\wedge \tau)|^2]^{1/2} \nonumber \\
      &\le~ C^{\ell_a} \ell_a! \cdot \tau^{-M} \sum_{v=1}^r\tau^{\alpha_v} \mathbb{E}^{\theta}[(s_0^+)^{2Mr}]^{1/(2r)}\prod_{i\neq v}\mathbb{E}^{\theta}[(s_0^+)^{2r \cdot \alpha_i}]^{1/(2r)}\nonumber\\
       &\le~ C^{k}k!\sum_{v=1}^r \tau^{\alpha_v-M} C(Mr,\theta)^{\frac{1}{2r}} \big(\prod_{i\neq v} C(r \alpha_i,\theta)\big)^{\frac{1}{2r}}n^{(M+\ell_a-\alpha_v)(2/3+\frac{\delta}{2}+\frac{\epsilon}{2})}\nonumber\\
       &\le~ C'(k,\ell_a,\theta,M)n^{\ell_a(2/3+\frac{\delta}{2}+\frac{\epsilon}{2})}n^{-M\frac{\epsilon}{2}}.\label{eqn: apply assumption}
    \end{align}
    

Combining \eqref{eqn: controlling I1I2} and \eqref{eqn: apply assumption} we bound \eqref{eqn: L2} by
\[C''(j,k,\theta,M)n^{(j+k)((2/3)+\delta+\epsilon)}n^{-M\frac{\epsilon}{2}}.\]
Choosing $M$ sufficiently large, depending on $\epsilon$, $j$, $k$, $\delta$,  and $K$, we find that the difference \eqref{eqn: difference} is indeed negligible.
\end{proof}

\subsection{Improved estimate for central moments}
\begin{lemma}\label{lemma: s to log Z}
Suppose 
\[
|t-n\psi_1(\theta)|\leq A n^{\frac{2}{3}}.
\]
Assuming the moment bounds \eqref{eqn: truncation assumption}, there are constants $C(k,\theta)$, locally bounded in $\theta$, such that, for $k\ge 2$ even
\begin{equation}\label{eqn: est}
\mathbb{E}[(\ov{\log Z_{n,t}^\theta})^k]\le C(k,\theta)n^{(1/3)k+(\delta/3)k}
\end{equation}
for all $n$ sufficiently large.  
\end{lemma}
\begin{proof}
The proof is by induction on $k$. For $k=2$, \eqref{eqn: est} holds with $\delta=0$. Assuming the estimate for even exponents less than $k$, we use the first expression in Theorem \ref{theorem: explicit formula general} to express the cumulant $\kappa_k(\log Z_{n,t}^\theta)$ as a sum of terms of the form
\begin{equation}\label{eqn: term-to-bound}
\prod_{B\in \pi} \mathbb{E}\left[(\ov{\log Z_{n,t}^\theta})^{a_{j,B}}H_{b_{j,B},t}(B_0(t))\right],
\end{equation}
where $\pi$ is a partition of $\{1,\ldots, k\}$ into $|\pi|$ blocks $B$, and $a_{j,B}+b_{j,B}=|B|$. 

Using Lemma \ref{lemma: expectation of log times hermite}, Corollary \ref{cor: exchange}, and Lemma \ref{lem: truncation} with $K> 2k$, we have, for $\tau=n^{2/3+\delta/2+\epsilon}$,
\begin{align*}
&~\prod_{B\in \pi} \mathbb{E}\left[(\ov{\log Z_{n,t}^\theta})^{a_{j,B}}H_{b_{j,B},t}(B_0(t))\right]\\
=&~\prod_{B\in \pi} \mathbb{E}\left[(\ov{\log Z_{n,t}^\theta})^{a_{j,B}}H_{b_{j,B},\tau}(B_0(\tau))\right]+O(n^{-k}).
\end{align*}
Taking absolute values and applying H\"older's inequality,
\begin{align*}
&\left|\mathbb{E}\left[(\ov{\log Z_{n,t}^\theta})^{a_{j,B}}H_{b_{j,B},\tau}(B_0(\tau))\right]\right|\\
\le&~\mathbb{E} [(\ov{\log Z_{n,t}^\theta})^k]^{\frac{a_{j,B}}{k}}\mathbb{E}[|H_{b_{j,B},\tau}(B_0(\tau))|^{k'}]^{\frac{b_{j,B}}{k'}}\\
\le&~ Cn^{((1/3)+\delta/4+\epsilon/2)b_{j,B}}\mathbb{E} [(\ov{\log Z_{n,t}^\theta})^k]^{\frac{a_{j,B}}{k}},
\end{align*}
where $\frac{a_{j,B}}{k}+\frac{1}{k'}=1$.  Taking the product over $B\in \pi$, we have, up to a constant factor, the bound:
    \begin{equation}\label{eqn: apply young}n^{((1/3)+\delta/4+\epsilon/2)b_j}\mathbb{E} [(\ov{\log Z_{n,t}^\theta})^k]^{\frac{a_j}{k}},
\end{equation}
where 
\[ a_j:=\sum_B a_{j,B}\quad \text{ and } \quad b_j:=\sum_B b_{j,B},\]
so $\frac{a_j}{k}+\frac{b_j}{k}=1$.  Note that for $1\leq j\leq k-1$, both $a_j,\,b_j\geq 1$.  Applying Young's inequality $xy\le \frac{1}{p}x^p+\frac{1}{q}y^q$ to \eqref{eqn: apply young}, we find that for $\eta>0$, any term of the form \eqref{eqn: term-to-bound} is bounded by
\[\eta \mathbb{E}[(\ov{\log Z_{n,t}^\theta})^k]+C(\eta)n^{((1/3)+\delta/4+\epsilon/2)k}+O(n^{-k}).\]
Combining this with Theorem \ref{theorem: explicit formula general}, we have
\begin{equation}\label{eqn: eta-sum}\kappa_k(\log Z_{n,t}^\theta)= C(k)\eta \mathbb{E}[(\ov{\log Z_{n,t}^\theta})^k]+C(k)C(\eta)n^{((1/3)+\delta/4+\epsilon/2)k}+O(n).
\end{equation}
Writing
\begin{equation}\label{eqn: cumulant-sum}
\kappa_k(\log Z_{n,t}^\theta)=\mathbb{E}[(\ov{\log Z_{n,t}^\theta})^k]+\sum_{\substack{|\alpha|=k\\0\leq\alpha_i<k}} c_\alpha \prod_{i=1}^{|\alpha|} \mathbb{E}[(\ov{\log Z_{n,t}^\theta})^{\alpha_i}],
\end{equation}
where the sum is over multi-indices $\alpha=(\alpha_1,\ldots,\alpha_k)$, $\sum_i \alpha_i=k$. If some $\alpha_i=k-1$, then the product must equal zero.  Therefore, by the induction assumption, all terms in the sum on the right of \eqref{eqn: cumulant-sum} are of order $n^{((1/3)+\delta/3)k}$. Choosing $\eta$ sufficiently small in \eqref{eqn: eta-sum} and absorbing $\epsilon/2$ into $\delta/4$, we obtain the result.
\end{proof}

\subsection{Finishing the argument}

Combining Corollary \ref{cor: log Z to s} and Lemma \ref{lemma: s to log Z} we obtain the following:

\begin{lemma}\label{lemma: inductive argument}
Suppose 
\begin{equation*}
    |t-n\psi_1(\theta_0)|\le An^{2/3}.
\end{equation*}
Assume there exist constants $\delta>0,\, L>1$, and $C(k)>0$ for $k\in \{2,4,\dots\}$, such that for any even $k$,
\[
\mathbb{E}[\ov{(\log Z_{n,t}^\theta})^k]\le C(k)n^{(1/3)k+\delta k}\quad \text{ for all } n\geq 1 \text{ and }  \theta\in [\theta_0,\theta_0+L].
\]

Then there exist constants $C'(k)>0$ for $k\in \{2,4,\dots\}$ such that for any even $k$,
\[
\mathbb{E}[(\ov{\log Z_{n,t}^\theta})^k]\le C'(k)n^{(1/3)k+(\delta/3) k}\quad \text{ for all } n\geq 1 \text{ and } \theta\in [\theta_0,\theta_0+L-1].
\]
\end{lemma}

Theorem \ref{thm: main-moments} will follow from repeated application of Lemma \ref{lemma: inductive argument} once we prove the following:

\begin{proposition}\label{prop: 1st step in induction}
For all  $\theta_0>0$ and $L>0$, there exist constants $C_k=C_k(\theta_0,L)>0$ for $k\in \{2,4,\dots\}$ such that for any even $k$,
\[
\mathbb{E}[(\ov{\log Z_{n,t}^\theta})^k]\le C_k n^{(1/3)k+(1/6) k}\quad \text{ for all } n\geq 1 \text{ and } \theta\in [\theta_0,\theta_0+L].
\]
\end{proposition}

\begin{proof}
For convenience, again let $A=\ov{\log Z_{n,t}^\theta}$, $B_0=B_0(t)$, and $R=\sum_{j=1}^n r_j^\theta(t)$.  By Proposition \ref{prop: burke}, $A=\overline{R}-B$. Thus, for even $k$,

\begin{equation}
    \E[A^k]\leq 2^{k-1}(\E[\ov{R}^k]+\E[B_0^k]).\label{eq: minkowsi bound}
\end{equation}
Since $R$ is a sum of i.i.d.\   random variables whose common distribution continuously depends on $\theta$,  there exist constants $C_k(\theta)>0$
, all of which are continuous in $\theta$, such that 
\[
\E[(\ov{R})^k]\leq C_k(\theta)n^{(k/2)} \quad \text{ for all } n\geq 1.
\]
The other expectation in \eqref{eq: minkowsi bound} satisfies
\[
\E[B_0^k]=(k/2-1)!!t^{(k/2)}\leq (k/2-1)!!\left(An^{(2/3)}+n\psi_1(\theta_0)\right)^{(k/2)}\leq D_k(\theta_0)n^{(k/2)}, 
\]
for all $n\geq 1$, where $D_k(\theta_0)>0$ are constants which are continuous in $\theta_0$.       Plugging these two inequalities into equation \eqref{eq: minkowsi bound} and using the continuity of $C_k(\theta)$ and $D_k(\theta_0)$ on $(0,\infty)$ yields the desired result.
\end{proof}

\begin{proof}[Proof of Theorem \ref{thm: main-moments}]
Let $\epsilon>0$, $\theta_0\in (0,\infty)$, and $p\in (0,\infty)$.  Fix even integers $k$, $M$ such that $p\leq k$ and
\[
\frac{(1/6)}{3^M}\leq \epsilon.
\]
By Jensen's inequality, it suffices to show the bounds \eqref{eq: Thrm 1-1} and \eqref{eq: Thrm 1-2} hold with $p$ replaced by $k$.  Now fix
  $L>M$ and apply Proposition \ref{prop: 1st step in induction} followed by $M$ consecutive applications of Lemma \ref{lemma: inductive argument} to obtain the bound \eqref{eq: Thrm 1-1}.  Finally, apply Corollary \ref{cor: log Z to s} to both $s_0^+$ and $s_0^-$ to obtain the bound \eqref{eq: Thrm 1-2}.
\end{proof}


\end{document}